\documentclass[reqno]{amsart}

\usepackage{mathrsfs}
\usepackage{amsmath}
\usepackage[all,cmtip,2cell]{xy}
\usepackage{amssymb}
\usepackage{fullpage}
\usepackage{microtype}
\usepackage{hyperref}
\usepackage{color}
\usepackage{array} 
 \usepackage[normalem]{ulem} 
 \usepackage{tikz}
 \usepackage{enumitem}
 
 \usepackage{parskip}
 
\def\Q{{\mathbb Q}}
\def\Z{{\mathbb Z}}

\def\C{{\mathbb C}}

\def\F{{\mathbb F}}
\def\G{{\mathbb G}}

\newcommand{\Uu}{{\mathscr U}}

\newcommand{\Bb}{{\mathscr B}}

\newcommand{\Xx}{{\mathscr X}}

\newcommand{\Hom}{{\rm Hom}}

\newcommand{\Sch}{{\rm Sch}}

\newcommand{\Spec}{{\rm Spec}}

\newcommand{\Aut}{{\rm Aut}}

\newcommand{\Sym}{{\rm Sym}}

\newcommand{\BG}{{\Bb G}}
\newcommand{\BT}{{\Bb T}}
\newcommand{\BB}{{\Bb B}}
\newcommand{\BU}{{\Bb U}}

\newcommand{\qa}{{{\overline{\Q}_l}}}

\def\Gal{\mathrm{Gal}}

\def\Aut{\mathrm{Aut}}
\def\Hom{\mathrm{Hom}}

\def\Spec{\mathrm{Spec}}

\def\Frob{\mathrm{Frob}}

\def\tr{\mathrm{tr}}

\def\dim{\mathrm{dim\,}}
\def\rk{\mathrm{rk}}
\def\deg{\mathrm{deg}}

\def\B{{B}}

\newtheorem{global-theorem}{Theorem}
\newtheorem{theorem}{Theorem}[section]

\newtheorem{proposition}[theorem]{Proposition}

\theoremstyle{definition}
\newtheorem{definition}[theorem]{Definition}
\newtheorem{example}[theorem]{Example}
\newtheorem{remark}[theorem]{Remark}

\begin{document}

\title{On the number of rational points of classifying stacks\\ for Chevalley group schemes}

\author[Balchin]{Scott Balchin}
\address[Balchin]{Warwick Mathematics Institute, Zeeman
  Building, Coventry, CV4 7AL, UK}
\author[Neumann]{Frank Neumann}
\address[Neumann]{School of Mathematics and Actuarial Science, Pure Mathematics Group, University of Leicester, University Road, Leicester, LE1 7RH, UK}
 \date{\today}
 
\maketitle

\begin{abstract}
We compute the number of rational points of classifying stacks of Chevalley group schemes using the Lefschetz-Grothendieck trace formula of Behrend for $\ell$-adic cohomology of algebraic stacks. From this we also derive associated zeta functions for these classifying stacks.
\end{abstract}

\section*{Introduction}

Given an algebraic group or group scheme $G$, we can associate to it a classifying stack $\BG$, which is an algebraic stack classifying principal $G$-bundles or $G$-torsors together with their automorphisms. Classifying stacks of algebraic groups play a similar role in algebraic geometry as classifying spaces of topological or Lie groups in algebraic topology. For example, classifying spaces of compact Lie groups and their singular cohomology encode a vast amount of geometry and topology, whereas classifying stacks of algebraic groups and their $\ell$-adic cohomology provide information about the geometry, arithmetic and representation theory of the algebraic group. Many cohomological results for compact Lie groups like Borel's fundamental theorems \cite{Bo} on the singular cohomology of classifying spaces, have direct analogs in terms of $\ell$-adic cohomology for classifying stacks of reductive algebraic groups. 

In this article we will be studying the rational $\ell$-adic cohomology of classifying stacks of a particular class of reductive algebraic groups over the field $\F_q$, namely the Chevalley groups.  It turns out that the numbers of $\F_q$-rational points of their associated classifying stacks are precisely the inverses of the orders of the related finite groups of Lie type, which are obtained as the finite groups of $\F_q$-rational points of the original Chevalley group. The calculation of $\F_q$-rational points employs the trace formula for algebraic stacks due to Behrend \cite{Be1}, \cite{Be2} involving the arithmetic Frobenius morphism acting on the $\ell$-adic cohomology of the associated classifying stack. These numbers of rational points are in fact groupoid cardinalities as they take into account all the aforementioned automorphisms. We calculate these numbers in all cases given by the classification, namely for the untwisted and twisted Chevalley groups. Once we have computed the number of $\mathbb{F}_q$-rational points of 
these classifying stacks we can assemble them into a zeta function for the classifying stack, which as usual encodes a lot of arithmetic information. We will derive these zeta function in all the cases coming out of the classification. It can be expected that the $\ell$-adic cohomology rings and the zeta functions of the classifying stacks will be of interest for the arithmetic and representation theory of Chevalley groups as well as for their associated finite groups of Lie type. Classifying spaces of compact Lie groups also have a rich homotopy theory and similarly algebraic stacks have \'etale homotopy types (see \cite{AM}, \cite{Fr}). The \'etale homotopy types of classifying stacks for reductive algebraic groups should have similar significance as the classical homotopy types for compact Lie groups.  We will address this circle of homotopical questions in a follow-up article to this work.

This article features the following content. In the first section we recall the definition of quotient stacks and classifying stacks. We also collect the required facts about $\ell$-adic cohomology of algebraic stacks. Furthermore, we describe the arithmetic Frobenius morphisms and its action on the cohomology and state Behrend's Lefschetz-Grothendieck trace formula, which counts the number of $\F_q$-rational points of algebraic stacks. Finally, we recall the zeta function for algebraic stacks and illustrate all the concepts for the fundamental example given by the multiplicative group scheme $\G_m$. In the second section we calculate the rational $\ell$-adic cohomology of classifying stacks for connected reductive algebraic groups over $\Spec(\F_q)$ in analogy with the classical results of Borel on the singular cohomology of classifying spaces for connected compact Lie groups \cite{Bo}. We state a general expression for the number of $\F_q$-rational points for these classifying stacks and derive as a corollary Steinberg's formula for the order of the finite group of $\F_q$-rational points of a given connective reductive algebraic group over $\F_q$ (see \cite{St}, \cite{St3}). In the final section, we apply this machinery to obtain explicit formulas for the number of $\F_q$-rational points and related zeta functions for the classifying stacks of Chevalley and Steinberg groups.

\section{Classifying stacks, $\ell$-adic cohomology, Frobenius morphisms and zeta functions}\label{Sec1}

In this section we will recall briefly the notions of classifying stacks of algebraic groups, their $\ell$-adic cohomology, and how to count $\F_q$-rational points using Behrend's Grothendieck-Lefschetz trace formula for algebraic stacks. For the general theory of algebraic stacks we will refer to the book by Laumon and Moret-Bailly \cite{LMB}. Other resources are \cite{So}, \cite{N}, \cite{Ol} or the encyclopedic stacks project \cite{Stacks}.

\begin{definition} Let $Z$ be a smooth scheme over a noetherian scheme $S$ and $G/S$ be a reductive group scheme of finite rank over $S$ together with a (right) $G$-action $\mu: Z\times_S G\rightarrow Z$. The {\it quotient stack} $[Z/G]$ is defined via its groupoid of sections as follows: For a given scheme $U/S$ over $S$, the groupoid $[Z/G](U)$ of $U$-valued points of $[Z/G]$ is the groupoid of principal $G$-bundles $P$ over $U$ together with a $G$-equivariant morphism $\alpha: P\rightarrow Z$ and isomorphisms of this data. In the special case that $Z=S$ being equipped with the trivial $G$-action, the resulting quotient stack $\BG:=[S/G]$  is called the {\it classifying stack} of $G$. 
\end{definition}

The quotient stack $[Z/G]$, and in particular the classifying stack $\BG$, under the above general conditions are smooth algebraic stacks which are locally of finite type (see \cite{LMB}, \cite{Be2} or \cite{N}).

The classifying stack $\mathscr{B}G$ should be viewed as an algebro-geometric analogue of the classifying space $BG$ in algebraic topology.  In the topological setting, principal $G$-bundles over a topological space $X$ are classified by homotopy classes of maps into the classifying space $BG$, whereas in algebraic geometry, principal $G$-bundles--or $G$-torsors--over a scheme $X$ can be classified by morphisms of stacks from $X$ into the classifying stack $\mathscr{B}G$.

We will now recall briefly the machinery of $\ell$-adic cohomology of algebraic stacks. As a general reference for the cohomology of algebraic stacks and its main properties we refer to the book of Laumon and Moret-Bailly \cite{LMB}. Especially for the definition of $\ell$-adic cohomology, its main properties and the general formalism of cohomology functors we refer to the work of Behrend \cite{Be1,Be2} and the subsequent articles by Laszlo and Olsson \cite{LO1,LO2}. An alternative systematic approach towards $\ell$-adic cohomology for algebraic stacks was given recently as well by Gaitsgory and Lurie \cite{GL}. Here we will follow also closely the exposition and notations of \cite{HeSch}.

Let $\Xx$ be an algebraic stack, which is smooth and locally of finite type over the base scheme $S=\Spec(\F_q)$, where $\F_q$ the field with $q=p^s$ elements. Let $\ell$ be a prime different from $p$. The rational $\ell$-adic cohomology $H^*( \Xx_{\overline{\F}_q}, \overline{\Q}_\ell)$ over the lisse-\'etale site $\Xx_{\text{lis-\'et}}$ of $\Xx$ is given as the limit of the $\ell$-adic cohomology groups over all open substacks $\Uu$ of finite type of the algebraic stack  $\Xx_{\overline{\F}_q}=\Xx\times_{\Spec(\F_q)}\Spec({\overline{\F}_q})$ over the algebraic closure $\overline{\F}_q$, i.e., we have
$$H^*( \Xx_{\overline{\F}_q}, \overline{\Q}_\ell)=\lim_{\substack{{\Uu\subset\Xx_{\overline{\F}_q}},\\ 
\text{open, finite type}}}\hspace*{-0.5cm}H^*(\Uu, \overline{\Q}_\ell).$$
In the following, we will normally write $H^*( \Xx, \overline{\Q}_\ell)$ instead of $H^*( \Xx_{\overline{\F}_q}, \overline{\Q}_\ell)$ to simplify the notations. 

The $\ell$-adic cohomology has a K\"{u}nneth decomposition~\cite{Be1,Be2,LO2}. That is, for $\mathscr{X}$, $\mathscr{Y}$ algebraic stacks as above, there is a natural isomorphism
 $$
 H^\ast(\mathscr{X} \times \mathscr{Y},\overline{\Q}_\ell) \simeq H^\ast(\mathscr{X}, \overline{\Q}_\ell) \otimes H^\ast(\mathscr{Y}, \overline{\Q}_\ell).
 $$

We also have an arithmetic Frobenius morphism, which acts on an algebraic stack $\Xx$ over $\Spec(\F_q)$ and its $\ell$-adic cohomology. For this, let
$$\Frob: \overline{\F}_q \rightarrow \overline{\F}_q, \,\, a\mapsto a^q$$
be the classical Frobenius morphism given by a generator of the Galois group $\Gal(\overline{\F}_q/\F_q)$ of the field extension $\overline{\F}_q/\F_q$. 
It induces an endomorphism of schemes
$$\Frob_{\Spec(\overline{\F}_q)}: \Spec(\overline{\F}_q) \rightarrow \Spec(\overline{\F}_q)$$
which naturally extends to an endomorphism of algebraic stacks defined as
$$\Frob_{\Xx}:=id_{\Xx} \times \Frob_{\Spec(\overline{\F}_q)}: \Xx_{\overline{\F}_q} \rightarrow \Xx_{\overline{\F}_q}.$$
This endomorphism $\Frob_{\Xx}$ of algebraic stacks is called the {\em arithmetic Frobenius morphism}.
By naturality, it induces an endomorphism $\Psi_q:=\Frob_{\Xx}^*$ on the $\ell$-adic cohomology of the algebraic stack $\Xx$:
$$\Psi_q=\Frob_{\Xx}^*: H^*(\Xx, \qa)\rightarrow H^*(\Xx, \qa).$$


Analogous to case of schemes $X$ over a base $S=\Spec(\F_q)$ we can calculate the number $\mathscr{X}(\mathbb{F}_q) = \# \mathscr{X}(\text{Spec}(\mathbb{F}_q)$ of $\F_q$-rational points for an algebraic stack $\Xx$ locally of finite type using a Lefschetz type trace formula. In the case of algebraic stacks though we need to modify this count to a 'stacky' count as we have to consider instead the groupoid cardinality of the groupoid $\Xx(\Spec(\F_q))$ of $\F_q$-rational points. 
The analogue of the Grothendieck-Lefschetz trace formula for algebraic stacks was conjectured and proved by Behrend, for particular cases in \cite{Be1}, and in full generality in \cite{Be2}.

\begin{theorem}[Behrend-Grothendieck-Lefschetz trace formula]\label{Trace}
Let $\Xx$ be a smooth algebraic stack of finite type over $\Spec(\F_q)$ and $\Psi$ be the arithmetic
Frobenius morphism. Then we have \vspace{-0.1cm}
$$q^{dim (\Xx)}\sum_{s\geq 0}(-1)^s \tr
(\Psi_q |H^s(\Xx, \overline{\Q}_\ell))= \#\Xx(\F_q).$$
Here $\#\Xx(\F_q)=\#\Xx(\Spec(\F_q))$ is the groupoid cardinality 
$$\#\Xx(\Spec(\F_q))=\sum_{x\in [\Xx(\Spec(\F_q))]}\frac{1}{\#Aut_{\Xx(\Spec(\F_q))}(x)}$$
of the groupoid $\Xx(\Spec(\F_q))$ of $\F_q$-rational points of the algebraic stack $\Xx$, where $\#Aut_{\Xx(\Spec(\F_q))}(x)$ is the order of the group of automorphisms of the isomorphism class $x$ in the groupoid $\Xx(\Spec(\F_q))$ and $[\Xx(\Spec(\F_q))]$ denotes the set of isomorphism classes of objects in the groupoid $\Xx(\Spec(\F_q))$.
\end{theorem}

We are mostly interested in algebraic groups over the field $\F_q$ and their classifying stacks. By an algebraic group we understand here a smooth group scheme $G$ of finite type over $\Spec(\F_q)$ (see \cite{M2}). The algebraic group $G$ is called connected if it is geometrically connected in the sense of  \cite[Exp. VI.A, 2.1.1]{DmGr}. In the case of quotient stacks of group actions of algebraic groups on schemes or algebraic spaces over $\F_q$, the groupoid cardinality of the groupoid of of $\F_q$-rational points of the quotient stack can be identified as follows  (see \cite[Prop. 2.2.3]{Be1})

\begin{proposition}\label{BGvsG}
Let $X$ be a smooth algebraic space of finite type over $\Spec(\F_q)$ and $G$ be a connected algebraic group over $\Spec(\F_q)$
acting on $X$. If $\Xx$ is the quotient stack $\Xx=[X/G]$, then we have
$$\#\Xx(\F_q)=\frac{\#X(\F_q)}{\#G(\F_q)}.$$
In particular, for the classifying stack $\BG$ of the algebraic group $G$ we have
$$\#\BG(\F_q)=\frac{1}{\#G(\F_q)}.$$
\end{proposition}

\begin{proof}This is essentially \cite[Lemma 2.5.1]{Be1} and the particular case follows by letting $X=\Spec(\F_q)$ be a point with trivial action of $G$. For the particular case, which is most relevant here, we can also argue directly in the following way. As $G$ is connected, Lang's theorem \cite{La} implies that each principal $G$-bundle on $\Spec(\F_q)$ is isomorphic to the trivial bundle and therefore the groupoid $\BG(\F_q)$ contains only a single object up to isomorphism, whose automorphism group is the finite group $G(\F_q)$. This gives the desired groupoid cardinality $\#\BG(\F_q)$.
\end{proof}

\begin{remark}
We could also employ a Lefschetz type trace formula for $\ell$-adic cohomology with compact support using the action of the geometric Frobenius and Poincar\'e duality instead (see \cite[Theorem 1.1]{Su}).  This approach has certain advantages due to the existence of a full theory of six-operations for algebraic stacks \cite{LO1, LO2} and removes any smoothness condition. We will however follow here the original approach of Behrend using the arithmetic Frobenius morphism \cite{Be1}, \cite{Be2}.
 \end{remark}
 
Finally, for an algebraic stack $\Xx$ of finite type over $\F_q$ we can also define a zeta function incorporating all the counts of $\F_{q^i}$-rational points formally defined as (see \cite{Be1})

\begin{definition}
Let $\mathscr{X}$ be a smooth algebraic stack of finite type over $\Spec(\F_q)$.  The \emph{zeta function} of $\mathscr{X}$ is the formal power series $\zeta_{\mathscr{X}}(t)\in \Q[[t]]$ given by
$$\zeta_{\mathscr{X}}(t) = \exp \left(  \sum^\infty_{i = 1}  \# \mathscr{X}(\mathbb{F}_{q^i}) \frac{t^i}{i} \right).$$
 \end{definition}
 
If $\mathscr{X}$ is a Deligne-Mumford stack, for example the moduli stack of elliptic curves $Ell$ or more generally the moduli stack $\mathscr{M}_{g, n}$ of algebraic curves of genus $g$ with $n$ marked points, it can be shown that $\zeta_{\mathscr{X}}$ is a rational function. But in general, this zeta function for algebraic stacks might not be a rational function. Nevertheless, Behrend \cite{Be1} showed that if $\mathscr{X}$ is given as a quotient stack of an algebraic space by a linear algebraic group, then the zeta function $\zeta_{\mathscr{X}}$ is a meromorphic function in the complex $t$-plane and can be calculated in many interesting cases
(compare also \cite{Su}).

Let us finish this section with an instructive guiding example, for which we can study all the ingredients introduced above and their interplay. This can be viewed as a prelude for the general case and our explicit calculations in the following sections (compare also \cite{Be2}, \cite{Su}).

\begin{example}[Multiplicative group scheme]
Let $\G_m$ be the multiplicative group scheme over  $\Spec(\F_q)$ i.e., given as
$$\G_m=\Spec(\F_q[x, x^{-1}]).$$

Let $\Bb \G_m$ be the associated classifying stack which classifies $\G_m$-torsors i.e., the classifying stack of line bundles. For the dimension of this algebraic stack we have
 $$\dim (\Bb \G_m) =- \dim(\G_m) =-1$$
The number $\#\Bb \G_m (\Spec(\F_q))$ of $\F_q$-rational points of  $\Bb \G_m$ is given as the number  of line bundles (up to isomorphism) over the ``point" $\Spec(\F_q)$. As all line bundles over the ``point" $\Spec(\F_q)$ are
trivial, there is therefore just one isomorphism class $x$ in $\Bb \G_m(\Spec(\F_q))$. 

Furthermore we have
$$\#\Aut_{\Bb \G_m(\Spec(\F_q))}(x)=\#\G_m(\F_q)=\#\F_q^*=q-1$$ 
so in other words we have that
$$\#\Bb \G_m(\F_q)=\sum_{x\in
[\Bb \G_m(\Spec(\F_q))]}\frac{1}{\#\Aut_{\Bb \G_m(\Spec(\F_q))}(x)}=\frac{1}{q-1}.$$

The $\ell$-adic  cohomology of $\Bb \G_m$ is basically the cohomology of an infinite
projective space, namely we have
$$H^*(\Bb \G_m, \overline{\Q}_\ell)=\overline{\Q}_\ell[c],$$ 
with $c$ a generator in degree 2 and so from the Behrend-Grothendieck-Lefschetz trace formula we therefore get
$$q^{\dim (\Bb \G_m)}\sum_{i\geq 0}\tr (\Psi_q |H^{2i}(\Bb \G_m,
\Q_\ell))=q^{-1}(1+q^{-1}+q^{-2}+\cdots)=\frac{1}{q}\sum_{i=0}^{\infty}\frac{1}{q^i}.$$
 
This formal calculation therefore gives a ``stacky" proof for the well-known formula
$$\sum_{i=0}^{\infty}\frac{1}{q^{i+1}}=\frac{1}{q-1}.$$
We can then compute the zeta function as follows
$$ 
\begin{aligned}
\zeta_{\Bb \G_m}(t)& =\exp(\sum_{i=1}^{\infty}\#\Bb \G_m (\F_{q^{i}})\frac{t^i}{i})=\exp(\sum_{i=1}^{\infty}\frac{1}{q^{i}-1}\frac{t^i}{i})\\
                        &=\exp(\sum_{i=1}^{\infty}\frac{t^i}{i}\sum_{k=1}^{\infty}\frac{1}{q^{ki}})=\prod_{k=1}^{\infty}\exp(\sum_{i=1}^{\infty}\frac{(t/q^{k})^i}{i})\\
                  &=\prod_{k=1}^{\infty}(1-q^{-k}t)^{-1}.
\end{aligned}
 $$
From this one can also derive a functional equation for the zeta function, namely we have
$$\zeta_{\Bb \G_m}(qt)=\frac{1}{1-t}\zeta_{\Bb \G_m}(t)$$
and so the zeta function for $\Bb \G_m$ has a meromorphic continuation to the complex $t$-plane having simple poles at $t=q^n$ for
all $n\geq 1$. 
\end{example}

\section{Cohomology of classifying stacks and the theorems of Borel and Steinberg}\label{sec:coh}

Let us collect some facts about the $\ell$-adic cohomology of classifying stacks, which are basically algebro-geometric analogues of classical theorems of Borel for classifying spaces of compact Lie groups \cite{Bo} (compare also \cite{Be1}, \cite{Be2}, \cite{Su} and \cite{IlZh}).

Let $G$ be a connected algebraic group over the field $\F_q$. We will employ the Leray spectral sequence of the universal morphism of algebraic stacks $f: \Spec(\F_q)\rightarrow \BG$ with fibre $G$, which is the algebro-geometric analogue of the universal topological fibration over the classifying space of a compact Lie group. The spectral sequence is given as follows:
$$E^{s, t}_2\cong H^s(\BG, R^tf_* \overline{\Q}_\ell)\Rightarrow H^{s+t}(\Spec(\F_q), \overline{\Q}_\ell).$$
Here $R^tf_* \overline{\Q}_\ell$ is a constructible sheaf on $\BG$ and we have $$R^tf_* \overline{\Q}_\ell=a^*f^*R^tf_*\overline{\Q}_\ell=a^*H^t(G, \overline{\Q}_\ell),$$ where $a: \BG\rightarrow \Spec(\F_q)$ is the structure morphism of the classifying stack $\BG$ and $H^t(G, \overline{\Q}_\ell)$ the $\Gal(\overline{\F}_q/\F_q)$-module interpreted as a sheaf on $\Spec(\F_q)$ (see \cite[Lemma 5.4]{Be2}).
The projection formula for the morphism $f$ of algebraic stacks then implies for the $E_2$-term of the spectral sequence that (see \cite[Cor. 5.3, Prop. 5.5]{Be2})
$$H^s(\BG, R^tf_* \overline{\Q}_\ell)\cong H^s(\BG, \overline{\Q}_\ell)\otimes_{\overline{\Q}_\ell}H^t(G, \overline{\Q}_\ell).$$
Furthermore, the spectral sequence converges to $H^*(\Spec(\F_q), \overline{\Q}_\ell)$, which means that we have $E^{s, t}_{\infty}=0$ if $(s, t)\neq (0, 0)$ and $H^0(\Spec(\F_q), \overline{\Q}_\ell)\cong \overline{\Q}_\ell$. Therefore the spectral sequence simplifies as follows
$$E^{s, t}_2\cong H^s(\BG, \overline{\Q}_\ell)\otimes_{\overline{\Q}_\ell}H^t(G, \overline{\Q}_\ell) \Rightarrow \overline{\Q}_\ell.$$

For each $t\geq 1$, the differential $d^{0, t}_{t+1}$ is an isomorphism
$$d^{0, t}_{t+1}: E^{0, t}_{t+1} \stackrel{\cong}\longrightarrow E^{t+1, 0}_{t+1}.$$
on the $(t+1)$-page. It is a monomorphism resp., epimorphism because it is the last possible non-zero differential {\it from} $E^{0, t}_*$
resp., {\it into} $E^{t+1, 0}_*$.
So for each $t\geq 1$ we get an isomorphism from the transgressive subspace $N^t=E^{0, t}_{t+1}$ of $H^t(G, \overline{\Q}_\ell) \cong E^{0, t}_2$ to the quotient $E^{t+1, 0}_{t+1}$ of $H^{t+1}(\BG, \overline{\Q}_\ell)\cong E^{t+1, 0}_2$. In particular, the epimorphism $H^{t+1}(\BG, \overline{\Q}_\ell)\twoheadrightarrow N^t$ is induced by the differential $d^{0, t}_{t+1}$.

Let now $N=\bigoplus_{t\geq 1}N^t$ be the graded transgressive $\overline{\Q}_\ell$-vector space. Then we have the following from Borel's transgression theorem \cite[Th\'eor\`eme 13.1]{Bo} in its algebraic form (compare \cite[Theorem 5.6]{Be2}, \cite[Theorem 4.8]{Su}):

\begin{itemize}
\item[(i)] If $t$ is even, then $N^t=0$.
\item[(ii)] The canonical map $\Lambda^* N \stackrel{\cong}\longrightarrow H^*(G, \overline{\Q}_\ell)$ is an isomorphism of graded $\overline{\Q}_\ell$-algebras.
\item[(iii)] The spectral sequence induces an epimorphism of graded $\overline{\Q}_\ell$-vector spaces
$$H^*(\BG, \overline{\Q}_\ell)\twoheadrightarrow N[-1]$$
and every section gives an isomorphism
$$\Sym^*(N[-1]) \stackrel{\cong}\longrightarrow H^*(\BG, \overline{\Q}_\ell).$$
\end{itemize}

Let us now assume that $G$ is a connected reductive algebraic group over the field $\F_q$ of rank $\rk(G)=n$. Let $B=T\cdot U$ be a Borel subgroup of $G$ with maximal torus $T$ and unipotent radical $U$. Let $N_G(T)$ be the normaliser of $T$ in $G$ and $W=N_G(T)/T$ the Weyl group. Furthermore, let $X_*(T)$ be the lattice of cocharacters $\lambda: \G_m\rightarrow T$ of $T$ and $\varepsilon'_1, \ldots \varepsilon'_n$ roots of unity given as the eigenvalues of the arithmetic Frobenius on the lattice $X_*(T)$ and therefore also on the character group  $X^*(T)$. From a classical theorem of Chevalley \cite{Che}, it follows that the subalgebra of $W$-invariant elements $\Sym^* (X_*(T)\otimes \C)^W$of the symmetric algebra $\Sym^* (X_*(T)\otimes \C)$ is generated by homogeneous polynomials $I_1, \ldots, I_n$ with uniquely determined degrees $d_1, \ldots, d_n$ i.e. $\deg(I_j)=d_j$ for $j=1, \ldots, n$. For example, for an $n$-dimensional torus $\G^n_m$ these degrees are simply $d_j=1$ for all $j=1, \ldots, n$. If $G$ is semi-simple, then we have $d_j>1$ for all $j=1, \ldots, n$. The homogeneous generators $I_j$ can be chosen as eigenvectors of the induced Frobenius action with their eigenvalues being roots of unity denoted by $\varepsilon_1, \ldots, \varepsilon_n$ (compare \cite[2.1]{St}, \cite[3.1]{KaRi}, \cite{Dm} and \cite[Sommes trig. 8.2]{SGA4.5}).

We have the projection morphism $p:\BT\rightarrow \BG$ with fibre the flag variety $G/T$ fitting into a $2$-cartesian diagram of algebraic stacks
\[
\xy \xymatrix{ G/T\ar[r]\ar[d]&
\BT\ar[d]^p\\
\Spec(\F_q)\ar[r]& \BG}
\endxy
\]

The above diagram plays a similar role as the analogous topological fibration for classifying spaces of compact Lie groups.
The following theorem is an algebro-geometric analogue of another classical theorem of Borel for connected compact Lie groups (compare \cite{Bo} and \cite{IlZh}).

\begin{theorem}\label{CohBG2}
Let $G$ be a connected reductive algebraic group over the field $\F_q$. Then we have:
\begin{itemize}
\item[(i)] The Leray spectral sequence for the projection morphism $p:\BT\rightarrow \BG$ is given as
$$E_2^{s,t}\cong H^s(\BG, \overline{\Q}_\ell)\otimes_{\overline{\Q}_\ell} H^t(G/T, \overline{\Q}_\ell)\Rightarrow H^{s+t}(\BT, \overline{\Q}_\ell)$$
and degenerates at the $E_2$-page. In particular, the induced homomorphism $$H^*(\BT, \overline{\Q}_\ell)\rightarrow H^*(G/T, \overline{\Q}_\ell)$$ is an epimorphism. In fact, $H^*(G/T, \overline{\Q}_\ell)$ is the regular representation of the Weyl group $W$ and generated by the Chern classes of those invertible 
sheaves ${\mathcal L}_\chi$ obtained as pushouts of the $T$-torsor $G\rightarrow G/T$ for the characters $\chi: T\rightarrow \G_m$.
\item[(ii)] The homomorphism $H^*(\BG, \overline{\Q}_\ell)\rightarrow H^*(\BT, \overline{\Q}_\ell)$ induced by the projection morphism $p:\BT\rightarrow \BG$ induces an isomorphism
$$H^*(\BG, \overline{\Q}_\ell)\cong H^*(\BT, \overline{\Q}_\ell)^W.$$
\item[(iii)] There is an isomorphism of graded $\overline{\Q}_\ell$-algebras
$$H^*(\BG, \overline{\Q}_\ell)\cong {\overline\Q}_\ell[c_1, \ldots c_n]$$
where the $c_i\in H^{2d_i}(\BG, \overline{\Q}_\ell)$ are Chern class generators in even degrees $2d_i$.
\item[(iv)] The arithmetic Frobenius morphism $\Psi_q=\Frob_{\BG}^*$ acts on the $\ell$-adic cohomology algebra as follows:
$$
\begin{aligned}
\Psi_q(c_i)&= \varepsilon_iq^{-d_i}
c_i
\end{aligned}
$$
\end{itemize}
\end{theorem}
\begin{proof}
These results can be derived from the analogous classical results by Borel for connected compact Lie groups by lifting $G$ to characteristic $0$ following the scholium of \cite[Sommes trig. 8.2.]{SGA4.5} (see also \cite[Example 4.11.]{IlZh}). The degeneration of the spectral sequence 
$$E_2^{s,t}\cong H^s(\BG, \overline{\Q}_\ell)\otimes_{\overline{\Q}_\ell} H^t(G/T, \overline{\Q}_\ell)\Rightarrow H^{s+t}(\BT, \overline{\Q}_\ell)$$
at $E_2$ results because $E_2^{s, t}=0$ if $s$ or $t$ is odd.  This follows from the Bruhat decomposition for the flag variety $G/B$ for a Borel subgroup $B$ containing $T$, which implies that $H^j(G/T, \overline{\Q}_\ell)=0$ if $j$ is odd. The spectral sequence inherits an action of the Weyl group $W$. It acts trivially on $H^*(\BG, \overline{\Q}_\ell)$ and via character theory as the regular representation on $H^*(G/T, \overline{\Q}_\ell)$. Taking the pushout of the principal $T$-bundle or $T$-torsor $T\to G \to G/T$ along a character $\chi \colon T \to \G_m$ as in the following diagram
$$\xymatrix{
T \ar[d] \ar[r]^\chi & \G_m \ar[d] \\ G \ar[r] \ar[d] & E \ar[d] \\ G/T \ar@{=}[r] & G/T
}$$
gives a line bundle over $G/T$ and the Chern classes of all these line bundles corresponding to such a character generate the cohomology $H^*(G/T, \overline{\Q}_\ell)$ (compare also \cite[Example 4.11(c)]{IlZh}.
Furthermore, the morphism $\BT\rightarrow \BB$ between classifying stacks induced by the inclusion of the maximal torus $T$ into the Borel subgroup $B$ has as fibres the classifying stack $\BU$ of the unipotent radical subgroup $U$ of $B$.  As $U$ is ${\mathbb A}^n$ all higher cohomology vanishes and we get an induced isomorphism in cohomology. The fibres of the morphism $\BB\rightarrow \BG$ are given by the flag varieties $G/B$ and therefore the morphism $\BT\rightarrow \BG$ induces a monomorphism 
$$H^*(\BG, \overline{\Q}_\ell)\rightarrow H^*(\BT, \overline{\Q}_\ell)$$
which lands inside the polynomial invariants of the Weyl group $W$ and for dimensional reasons we get an isomorphism 
$$H^*(\BG, \overline{\Q}_\ell)\cong H^*(\BT, \overline{\Q}_\ell)^W.$$
The statement on the polynomial generators for the cohomology of $\BG$ and the action of the arithmetic Frobenius morphism holds first of all for tori. This follows using the K\"unneth theorem and because $H^*(\B\G_m, \overline{\Q}_\ell)\cong \overline{\Q}_\ell [c_1]$ with $\deg(c_1)=2$, $d_1=1$ and action of the arithmetic Frobenius $\Psi_q(c_1)=\varepsilon_1q^{-1}$. If $T$ is a maximal torus, i.e. $T\cong \G_m^n$, then for the cohomology algebra of the classifying stack we have $H^*(\BT, \overline{\Q}_\ell)\cong \overline{\Q}_\ell[t_1, \ldots t_n]$, where the generators $t_i$ have degrees $\deg(t_i)=2$ for all $i=1, \ldots n$. The Weyl group invariants $H^*(\BT, \overline{\Q}_\ell)^W$ form itself a polynomial algebra $\overline{\Q}_\ell[c_1, \ldots c_n]$ on homogeneous generators of degrees $2d_i$. 
Using the fact that the degrees $d_i$ for $i=1, \ldots, n$ of the homogeneous polynomial generators of the Weyl group invariants are precisely the degrees of the homogeneous generators of 
$H^*(\BT, \overline{\Q}_\ell)^W$ then implies the desired result on the cohomology of the classifying stack $\BG$ with the described action of the arithmetic Frobenius morphism (compare also \cite[9.1.4]{De2}, \cite{HeSch}).
\end{proof}

This allows us now to calculate the number of $\F_q$-rational points for the classifying stack of a given connected reductive algebraic group using the Behrend-Grothendieck-Lefschetz trace formula, which will be instrumental for our explicit calculations in the case of Chevalley groups.

\begin{theorem}\label{BGpoints}
Let $G$ be a connected reductive algebraic group over the field $\F_q$ of rank $\rk(G)=n$. Then the number of $\F_q$-rational points of $\BG$
is given as
$$\#\BG(\F_q)=q^{-\dim(G)} \prod_{i=1}^n {(1 -\varepsilon_iq^{-d_i}})^{-1}.$$
\end{theorem}

\begin{proof} We have $\dim(\BG)=-\dim(G)$ for the classifying stack $\BG$ of $G$.  From the Behrend-Grothendieck-Lefschetz trace formula \ref{Trace} we obtain from \ref{CohBG2} using the explicit eigenvalues of the arithmetic Frobenius morphism for an adequate choice of a basis
$$
\begin{aligned}
\#\BG(\F_q) &= q^{\dim (\BG)}\sum_{s\geq 0}(-1)^s \tr(\Psi_q |H^s(\BG, \overline{\Q}_\ell))\\
&= q^{-\dim(G)}\sum_{s\geq 0}\tr(\Psi_q |H^{2s}(\BG, \overline{\Q}_\ell))\\
&=q^{-\dim(G)}\prod_{i=1}^n (1 +\varepsilon_i q^{-d_i} + (\varepsilon_i q^{-{d_i}})^2 +(\varepsilon_i q^{-{d_i}})^3 +\cdots)\\
&=q^{-\dim(G)} \prod_{i=1}^n {(1 -\varepsilon_iq^{-d_i}})^{-1}.
\end{aligned}
$$
This is the desired formula for the number of $\F_q$-rational points of the classifying stack.
\end{proof}

As a corollary we get immediately the following theorem of Steinberg (see \cite[11.16]{St} or \cite[Theorem 25]{St3}), which calculates the order of the finite group $G(\F_q)$ of $\F_q$-rational points of the algebraic group $G$.

\begin{theorem}[Steinberg]\label{Steinberg}
Let $G$ be a connected reductive algebraic group over the field $\F_q$ of rank $\rk(G)=n$. Then the number of $\F_q$-rational points of $G$
is given as
$$\#G(\F_q)= q^{\dim (G)}\prod_{i=1}^n (1-\varepsilon_i q^{-d_i}).$$
\end{theorem}

\begin{proof} The formula follows at once from \ref{BGpoints} using \ref{BGvsG} 
$$\#G(\F_q)= \frac{1}{\#\BG(\F_q)}=q^{\dim (G)}\prod_{i=1}^n (1-\varepsilon_i q^{-d_i}).$$
\end{proof}

In the final section we will now use these formulas to calculate the number of $\F_q$-rational points and the related zeta functions for classifying stacks of certain algebraic groups intimately related to finite groups of Lie type. 

\section{Applications to Chevalley schemes and finite groups of Lie type}\label{sec:chev}

Let $G$ be a connected compact Lie group. Associated to $G$ is a reductive complex algebraic group $G(\C)$, the complexification of $G$, which can be constructed as the algebraic group of $\C$-rational points of a group scheme $G_{\C}$ over $\C$ obtained via base change from the associated integral affine Chevalley group scheme $G_\Z$, i.e.,
$$G_{\C}=G_{\Z}\times_{\Spec(\Z)}\Spec(\C).$$
In fact, for any field $k$, taking the $k$-rational points of the Chevalley group scheme $G_k$ over $k$ given via base change as
$$G_k=G_{\Z}\times_{\Spec(\Z)}\Spec(k)$$
we obtain the Chevalley group
$$G(k)=\Hom_{\Sch/k}(\Spec(k), G_k),$$
where $\Sch/k$ is the category of schemes over the field $k$. 

We are interested here in particular in the finite Chevalley group $G(\F_q)$
of $\F_q$-rational points, which can be constructed also as the fixed point set 
$$G(\F_q)=G(\overline{\F}_q)^{\psi_q}$$
of the Frobenius morphism
$$\psi_q: G(\overline{\F}_q)\rightarrow G(\overline{\F}_q).$$
Here $\psi_q$ is the arithmetic Frobenius homomorphism induced by the classical Frobenius homomorphism$$\Frob: \overline{\F}_q \rightarrow \overline{\F}_q, \,\, a\mapsto a^q.$$ See the references \cite{Dm, DmGr,  M2} for more information on the general theory of reductive group schemes.

The goal of this section is to compute the number of $\F_q$-rational points and the zeta functions for the classifying stacks of Chevalley group schemes over $\F_q$. Recall that such groups are determined by their complex dimension and the degrees of the fundamental Weyl group invariants (see for example \cite{Ca1}). We summarise the necessary data in Table \ref{tab1}. Our first aim is to rewrite Steinberg's theorem (Theorem~\ref{Steinberg}) in terms of this data.

 \begin{table}[h]
\centering
\label{fivs}
\begin{tabular}{|c|c|c|@{}m{0pt}@{}}
\hline
{\bf  Group} & {\bf  $\text{Dim}/\mathbb{C}$} & {\bf  Degrees of Weyl group invariants}&\\  \hline
$SL_{n+1}$ (Type $A_n$) & $n(n+2)$ & $2,3,\dots,n+1$ & \\[10pt] \hline 
$SO_{2n+1}$ (Type $B_n$)  & $n(2n+1)$ & $2,4, \dots , 2n$ & \\[10pt] \hline
$Sp_{2n}$ (Type $C_n$) & $n(2n+1)$ & $2,4 \dots , 2n$ & \\[10pt] \hline
$SO_{2n}$ (Type $D_n$) & $n(2n-1)$ &  $n, 2, 4, \dots , 2n-2$ & \\[10pt] \hline
$G_2$ & 14 &       $2,6$ &  \\[10pt] \hline
$F_4$ & 52 &     $2,6,8,12$ & \\[10pt] \hline
$E_6$ & 78 &     $2, 5, 6, 8, 9, 12$ & \\[10pt] \hline
$E_7$ & 133 &      $2, 6, 8, 10, 12, 14, 18$ & \\[10pt] \hline
$E_8$ & 248 &        $2, 8, 12, 14, 18, 20, 24, 30$ & \\[10pt] \hline
\end{tabular}\caption{The Weyl group data for the Chevalley groups.}\label{tab1}
\end{table}
 
 \begin{proposition}\label{chevalleycorrollary}
 Let $G$ be a Chevalley group scheme over $\mathbb{F}_q$. Then we have
 $$ \# \mathscr{B}G(\mathbb{F}_q) = \sum_{i_1=0}^\infty \cdots \sum_{i_r=0}^\infty q^{-(i_1d_1 + \cdots + i_rd_r + \dim(G))},$$
 where the $d_i$ are the degrees of the Weyl group invariants and $r$ is the rank of $G$.
 \end{proposition}
 
 \begin{proof}
We start by noting that these Chevalley group schemes are connected reductive algebraic groups over $\mathbb{F}_q$. Therefore, from Theorem~\ref{CohBG2} we have isomorphisms $H^*(\BG, \overline{\Q}_\ell)\cong H^*(\BT, \overline{\Q}_\ell)^W\cong {\overline\Q}_\ell[c_1, \ldots c_r]$. In particular the cohomology generators $c_i$ are in degree $2 d_i$, where the $d_i $ are the degrees of the Weyl group invariants. Comparing this with Theorem~\ref{BGpoints}, we see that these degrees of the Weyl group invariants are exactly the $d_i$ as in Steinberg's Theorem~\ref{Steinberg}. The result then follows from~\ref{CohBG2} as we can identify the action of the arithmetic Frobenius on the $\ell$-adic cohomology algebra of the classifying stack $\BG$ in terms of the degrees $d_i$ of the Weyl group invariants.
 \end{proof}
 
\begin{example}\label{compexample}
Let $G$ be the exceptional group $G_2$.  We can use Proposition~\ref{chevalleycorrollary} to easily compute $\# \mathscr{B}G_2 (\mathbb{F}_q)$ using the fact that $\dim(G) = 14$ and it has Weyl group invariants in degrees 2 and 6 respectively:
$$\#\mathscr{B}G_2(\mathbb{F}_q) =  \sum^\infty_{i_1=0} \sum^\infty_{i_2=0} q^{-(2i_1+6i_2+14)}.$$
One can computationally check that this infinite sum does indeed converge to 
$$\left(q^{14} (1-q^{-2})(1-q^{-6}) \right)^{-1} =  \#G_2(\mathbb{F}_q) ^{-1}$$
as expected.
\end{example}

From Proposition~\ref{chevalleycorrollary}, along with some manipulation of complex functions, we can immediately give a general form for the zeta function of the classifying stack.

 \begin{proposition}\label{chevalleycorrollary2}
 Let $G$ be a Chevalley group scheme over $\mathbb{F}_q$. Then
$$\zeta_{\mathscr{B}G(\mathbb{F}_q)}(t) = \prod_{k_1=1}^\infty \cdots \prod_{k_r=1}^\infty \left(1-q^{-(k_1d_1 + \cdots + k_rd_r + \left(\dim(G)-\Sigma_{i=1}^r d_i)\right)} t \right)^{-1},$$
 where the $d_i$ are the degrees of the Weyl group invariants.
 \end{proposition}

\begin{proof}
This result is an exercise in manipulating complex functions. Recall that we have
$$\zeta_{\mathscr{X}}(t) = \exp \left(  \sum^\infty_{i = 1}  \# \mathscr{X}(\mathbb{F}_{q^i}) \frac{t^i}{i} \right).$$
We then use the result of Proposition~\ref{chevalleycorrollary} to substitute in the value of $\# \mathscr{X}(\mathbb{F}_{q^i})$
 $$= \exp  \left(  \sum^\infty_{i = 1}  \sum_{k_1=0}^\infty \cdots \sum_{k_r=0}^\infty q^{-i(k_1d_1 + \cdots + k_rd_r + \dim(G))} \frac{t^i}{i} \right),$$
Next we adjust the indices $k_s$ so they start at $1$ instead of $0$
$$= \exp  \left(  \sum^\infty_{i = 1}  \sum_{k_1=1}^\infty \cdots \sum_{k_r=1}^\infty q^{-i(k_1d_1 + \cdots + k_rd_r + \dim(G) - \sum^r_{s=1} d_s)} \frac{t^i}{i} \right).$$
Finally, we follow the manipulation for the case of $\mathbb{G}_m$ to get the desired result. In particular we have
\begin{align*}
= & \prod_{k_1=1}^\infty \cdots \prod_{k_r=1}^\infty \exp (\sum^\infty_{i=1} q^{-i(k_1d_1 + \cdots + k_rd_r + \dim(G) - \sum^r_{s=1} d_s)} \frac{t^i}{i})\\
=& \prod_{k_1=1}^\infty \cdots \prod_{k_r=1}^\infty \left(1-q^{-(k_1d_1 + \cdots + k_rd_r + \left(\dim(G)-\Sigma_{s=1}^r d_s)\right)} t \right)^{-1},
\end{align*}
which is the desired expression.
\end{proof}

%
%

We now use Proposition~\ref{chevalleycorrollary} and Proposition~\ref{chevalleycorrollary2} along with the datum of Table~\ref{tab1} to compute systematically the number of $\mathbb{F}_q$-rational points of the classifying stacks $\BG$ of all the classical, exceptional and twisted Chevalley group schemes $G$. We also give the corresponding zeta functions. The groups $G(\F_q)$ of $\F_q$-rational points of $G$ are finite groups of Lie type and their orders are well known and can be found for example in~\cite[\S 1.5]{Hum}. From these orders, we can then compute the invariants needed for the computation of the numbers of $\mathbb{F}_q$-rational points of the classifying stacks and the corresponding zeta functions. 

\subsection{Classical and exceptional Chevalley groups}

We will first discuss the classical and exceptional Chevalley groups, as constructed by Chevalley and which give rise to the untwisted family of finite groups of Lie type (see \cite{Che1}, \cite{Ca2}).

\subsubsection{Classical Chevalley Groups}\leavevmode

\underline{$G=SL_{n+1}$}

\begin{align*}
\# G (\mathbb{F}_q) &=q^{n(n+2)}(1-q^{-2})(1-q^{-3}) \cdots (1-q^{-(n+1)})\\
\#\mathscr{B}G(\mathbb{F}_q) &= \sum^\infty_{i_1=0} \sum^\infty_{i_2=0} \cdots \sum^\infty_{i_n=0} q^{-(2i_1 + 3i_2 + \cdots + (n+1)i_n + n(n+2))}\\
\zeta_{\mathscr{B}G}(t) &= \prod_{k_1=1}^\infty \prod_{k_2=1}^\infty \cdots \prod_{k_n=1}^\infty \left(1-q^{-\left(2k_1 + 3k_2 + \cdots + (n+1)k_n + \frac{n(n+1)}{2} \right)} t \right)^{-1}
\end{align*}

\underline{$G=SO_{2n+1}$}

\begin{align*}
\# G (\mathbb{F}_q) &=q^{n(2n+1)}(1-q^{-2})(1-q^{-4}) \cdots (1-q^{-2n})\\
\#\mathscr{B}G(\mathbb{F}_q) &=  \sum^\infty_{i_1=0} \sum^\infty_{i_2=0} \cdots \sum^\infty_{i_n=0} q^{-(2i_1 + 4i_2 + \cdots + 2ni_n + n(2n+1))}\\
\zeta_{\mathscr{B}G}(t) &= \prod_{k_1=1}^\infty \prod_{k_2=1}^\infty \cdots \prod_{k_n=1}^\infty \left(1-q^{-\left(2k_1 + 4k_2 + \cdots + 2nk_n +  n^2 \right)} t \right)^{-1}
\end{align*}

\underline{$G=Sp_{2n}$}

\begin{align*}
\# G (\mathbb{F}_q) &=q^{n(2n+1)}(1-q^{-2})(1-q^{-4}) \cdots (1-q^{-2n})\\
\#\mathscr{B}G(\mathbb{F}_q) &=  \sum^\infty_{i_1=0} \sum^\infty_{i_2=0} \cdots \sum^\infty_{i_n=0} q^{-(2i_1 + 4i_2 + \cdots + 2ni_n + n(2n+1))}\\
\zeta_{\mathscr{B}G}(t) &= \prod_{k_1=1}^\infty \prod_{k_2=1}^\infty \cdots \prod_{k_n=1}^\infty \left(1-q^{-\left(2k_1 + 4k_2 + \cdots + 2nk_n  + n^2 \right)} t \right)^{-1}
\end{align*}

\underline{$G=SO_{2n}$}

\begin{align*}
\# G (\mathbb{F}_q) &=q^{n(2n-1)}(1-q^{-n})(1-q^{-2})(1-q^{-4})  \cdots (1-q^{-(2n-2)})\\
\#\mathscr{B}G(\mathbb{F}_q) &=  \sum^\infty_{i_1=0} \sum^\infty_{i_2=0} \cdots \sum^\infty_{i_n=0} q^{-(ni_1 + 2i_2 + 4i_3 + \cdots + (2n-2)i_n + n(2n-1))}\\
\zeta_{\mathscr{B}G}(t) &= \prod_{k_1=1}^\infty \prod_{k_2=1}^\infty \cdots \prod_{k_n=1}^\infty \left(1-q^{-\left(nk_1 + 2k_2 + 4k_3 + \cdots + (2n-2)k_n + n(n-1) \right)} t \right)^{-1}
\end{align*}

\subsubsection{Exceptional Chevalley Groups}\leavevmode

\underline{$G=G_2$}

\begin{align*}
\# G (\mathbb{F}_q) &= q^{14}(1-q^{-2})(1-q^{-6})\\
\#\mathscr{B}G(\mathbb{F}_q) &=  \sum^\infty_{i_1=0} \sum^\infty_{i_2=0} q^{-(2i_1+6i_2+14)}\\
\zeta_{\mathscr{B}G}(t) &= \prod_{k_1=1}^\infty \prod_{k_2=1}^\infty \left(1-q^{-(2k_1+6k_2+6)}t \right)^{-1}
\end{align*}

\underline{$G=F_4$}

\begin{align*}
\# G (\mathbb{F}_q) &= q^{52}(1-q^{-2})(1-q^{-6})(1-q^{-8})(1-q^{-12})\\
\#\mathscr{B}G(\mathbb{F}_q) &=\sum^\infty_{i_1=0} \sum^\infty_{i_2=0} \sum^\infty_{i_3=0} \sum^\infty_{i_4=0} q^{-(2i_1+6i_2+ 8i_3 + 12i_4 + 52)}\\
\zeta_{\mathscr{B}G}(t) &= \prod_{k_1=1}^\infty \prod_{k_2=1}^\infty \prod_{k_3=1}^\infty \prod_{k_4=1}^\infty \left(1-q^{-(2k_1+6k_2+8k_3+12k_4+24)} t \right)^{-1}
\end{align*}

\underline{$G=E_6$}

\begin{align*}
\# G (\mathbb{F}_q) &= q^{78}(1-q^{-2})(1-q^{-5})(1-q^{-6})(1-q^{-8})(1-q^{-9})(1-q^{-12})\\
\#\mathscr{B}G(\mathbb{F}_q) &= \sum^\infty_{i_1=0} \sum^\infty_{i_2=0} \cdots \sum^\infty_{i_6=0} q^{-(2i_1+5i_2+6i_3+8i_4+9i_5+12i_6 + 78)}\\
\zeta_{\mathscr{B}G}(t) &= \prod_{k_1=1}^\infty \prod_{k_2=1}^\infty \cdots \prod_{k_6=1}^\infty \left(1-q^{-\left(2k_1+5k_2+6k_3+8k_4+9k_5+12k_6 + 36 \right)} t \right)^{-1}
\end{align*}

\underline{$G=E_7$}

\begin{align*}
\# G (\mathbb{F}_q) &= q^{133}(1-q^{-2})(1-q^{-6})(1-q^{-8})(1-q^{-10})(1-q^{-12})(1-q^{-14})(1-q^{-18})\\
\#\mathscr{B}G(\mathbb{F}_q) &=  \sum^\infty_{i_1=0} \sum^\infty_{i_2=0} \cdots \sum^\infty_{i_7=0} q^{-(2i_1+6i_2+8i_3+10i_4+12i_5+14i_6 + 18i_7 + 133)}\\
\zeta_{\mathscr{B}G}(t) &= \prod_{k_1=1}^\infty \prod_{k_2=1}^\infty \cdots \prod_{k_7=1}^\infty \left(1-q^{-\left(2k_1+6k_2+8k_3+10k_4+12k_5+14k_6 + 18k_7 + 63 \right)} t \right)^{-1}
\end{align*}

\underline{$G=E_8$}

\begin{align*}
\# G (\mathbb{F}_q) &= q^{248}(1-q^{-2})(1-q^{-8})(1-q^{-12})(1-q^{-14})(1-q^{-18})(1-q^{-20})(1-q^{-24})(1-q^{-30})\\
\#\mathscr{B}G(\mathbb{F}_q) &=   \sum^\infty_{i_1=0} \sum^\infty_{i_2=0} \cdots \sum^\infty_{i_8=0} q^{-(2i_1+8i_2+12i_3+14i_4+18i_5+20i_6 + 24i_7 + 30i_8 + 248)}\\
\zeta_{\mathscr{B}G}(t) &= \prod_{k_1=1}^\infty \prod_{k_2=1}^\infty \cdots \prod_{k_8=1}^\infty \left(1-q^{-\left(2k_1+8k_2+12k_3+14k_4+18k_5+20k_6 + 24k_7 + 30k_8+ 120 \right)} t \right)^{-1}
\end{align*}

\subsection{Twisted Chevalley or Steinberg Groups}\label{Sec3}

Finally,  we shall look at the twisted Chevalley groups, which are sometimes also referred to as the Steinberg groups and which were constructed by Steinberg \cite{St1} as a variation and generalisation of Chevalley's original construction and Lang's theorem. They give rise to the twisted family of finite groups of Lie type \cite{Ca2}. The twisted Chevalley groups are given as algebraic group schemes over $\mathbb{F}_q$ for particular prime powers $q$, and therefore fall into the context given by Steinberg's theorem (see~\cite{Hum}). These twisted groups can be classified as follows:
 
 \begin{itemize}
 \item Classical Steinberg groups \cite{St1}:
 	\begin{enumerate}
	\item $~^2A_n$ over $\mathbb{F}_q$ with $q = p^{2k}$.
	\item $~^2D_n$ over $\mathbb{F}_q$ with $q = p^{2k}$.
	\end{enumerate}
 \item Exceptional Steinberg groups \cite{St1}:
 	\begin{enumerate}[resume]
	\item $~^2 E_6$ over $\mathbb{F}_q$ with $q=p^{2k}$ (constructed also independently by Tits \cite{Ti}).
	\item $~^3D_4$ over $\mathbb{F}_q$ with $q=p^{3k}$.
	\end{enumerate}
\end{itemize}

Note that the \emph{twisting} implies that the roots of unity $\epsilon_i$ appearing in the theorem of Steinberg are not always $1$ here, which can be seen from the calculation of orders of the finite groups of $\F_q$-rational points.

\subsubsection{Classical Steinberg Groups}\leavevmode

\underline{$G=~^2A_n$ over $\F_q$ with $q=p^{2k}$}

$$\# G (\mathbb{F}_q) = q^{n(n+2)}(1+q^{-2})(1-q^{-3}) \cdots (1-(-1)^{n+1}q^{-(n+1)})$$

\underline{For $n=4a$}

\begin{align*}
\#\mathscr{B}G(\mathbb{F}_{q}) &=  \sum^\infty_{i_1=0} \sum^\infty_{i_2=0} \cdots \sum^\infty_{i_n=0} (-1)^{(i_1 + i_3 + \cdots + i_{n-1})} q^{-(2i_1 + 3i_2 + \cdots + (n+1)i_n + n(n+2))}\\
\zeta_{\mathscr{B}G}(t) &= \prod_{k_1=1}^\infty \prod_{k_2=1}^\infty \cdots \prod_{k_n=1}^\infty \left(1-(-1)^{(k_1 + k_3 + \cdots + k_{n-1})}q^{-\left(2k_1 + 3k_2 + \cdots + (n+1)k_n + \frac{n(n-1)}{2} -2\right)} t \right)^{-1}
\end{align*}

\underline{For $n=4a+1$}

\begin{align*}
\#\mathscr{B}G(\mathbb{F}_{q}) & =  \sum^\infty_{i_1=0} \sum^\infty_{i_2=0} \cdots \sum^\infty_{i_n=0} (-1)^{(i_1 + i_3 + \cdots + i_n)} q^{-(2i_1 + 3i_2 + \cdots + (n+1)i_n + n(n+2))}\\
\zeta_{\mathscr{B}G}(t) &= \prod_{k_1=1}^\infty \prod_{k_2=1}^\infty \cdots \prod_{k_n=1}^\infty \left(1+(-1)^{(k_1 + k_3 + \cdots + k_n)}q^{-\left(2k_1 + 3k_2 + \cdots + (n+1)k_n + \frac{n(n-1)}{2} -2\right)} t \right)^{-1}
\end{align*}

\underline{For $n=4a+2$}

\begin{align*}
\#\mathscr{B}G(\mathbb{F}_{q}) &=  \sum^\infty_{i_1=0} \sum^\infty_{i_2=0} \cdots \sum^\infty_{i_n=0} (-1)^{(i_1 + i_3 + \cdots + i_{n-1})} q^{-(2i_1 + 3i_2 + \cdots + (n+1)i_n + n(n+2))}\\
\zeta_{\mathscr{B}G}(t) &= \prod_{k_1=1}^\infty \prod_{k_2=1}^\infty \cdots \prod_{k_n=1}^\infty \left(1+(-1)^{(k_1 + k_3 + \cdots + k_{n-1})}q^{-\left(2k_1 + 3k_2 + \cdots + (n+1)k_n + \frac{n(n-1)}{2} -2\right)} t \right)^{-1}
\end{align*}

\underline{For $n=4a+3$}

\begin{align*}
\#\mathscr{B}G(\mathbb{F}_{q}) &=  \sum^\infty_{i_1=0} \sum^\infty_{i_2=0} \cdots \sum^\infty_{i_n=0} (-1)^{(i_1 + i_3 + \cdots + i_n)} q^{-(2i_1 + 3i_2 + \cdots + (n+1)i_n + n(n+2))}\\
\zeta_{\mathscr{B}G}(t) &= \prod_{k_1=1}^\infty \prod_{k_2=1}^\infty \cdots \prod_{k_n=1}^\infty \left(1-(-1)^{(k_1 + k_3 + \cdots + k_n)}q^{-\left(2k_1 + 3k_2 + \cdots + (n+1)k_n + \frac{n(n-1)}{2} -2\right)} t \right)^{-1}
\end{align*}

\underline{$G=~^2D_n$ over $\mathbb{F}_{q}$ with $q=p^{2k}$}

\begin{align*}
\# G (\mathbb{F}_q) &=q^{n(2n-1)}(1+q^{-n})(1-q^{-2})(1-q^{-4})  \cdots (1-q^{-(2n-2)})\\
\#\mathscr{B}G(\mathbb{F}_{q}) &=  \sum^\infty_{i_1=0} \sum^\infty_{i_2=0} \cdots \sum^\infty_{i_n=0} (-1)^{i_1} q^{-(ni_1+2i_2 + 4i_3 + \dots + (2n-2)i_n + n(2n-1))}\\
\zeta_{\mathscr{B}G}(t) &= \prod_{k_1=1}^\infty \prod_{k_2=1}^\infty \cdots \prod_{k_n=1}^\infty \left(1 + (-1)^{k_1} q^{-\left(nk_1 + 2k_2 + 4k_3 + \cdots + (2n-2)k_n + n(n-1) \right)} t \right)^{-1}
\end{align*}

\subsubsection{Exceptional Steinberg Groups}\leavevmode

\underline{$G=~^2E_6$ over $\mathbb{F}_{q}$ with $q=p^{2k}$}

\begin{align*}
\# G (\mathbb{F}_q) &= q^{78}(1-q^{-2})(1+q^{-5})(1-q^{-6})(1-q^{-8})(1+q^{-9})(1-q^{-12})\\
\#\mathscr{B}G(\mathbb{F}_{q}) &=  \sum^\infty_{i_1=0} \sum^\infty_{i_2=0} \cdots \sum^\infty_{i_6=0} (-1)^{i_2+i_5}q^{-(2i_1+5i_2+6i_3+8i_4+9i_5+12i_6+78)}\\
\zeta_{\mathscr{B}G}(t) &= \prod_{k_1=1}^\infty \prod_{k_2=1}^\infty \cdots \prod_{k_6=1}^\infty \left(1- (-1)^{(k_2+k_5)}q^{-\left(2k_1+5k_2+6k_3+8k_4+9k_5+12k_6 + 36 \right)} t \right)^{-1}
\end{align*}

\underline{$G=~^3D_4$ over $\mathbb{F}_{q}$ with $q=p^{3k}$}

\begin{align*}
\# G (\mathbb{F}_q) &= q^{28}(1-q^{-2})(1-\xi q^{-4})(1-\xi^2 q^{-4})(1-q^{-6}) \text{ with } \xi^3=1, \xi \neq 1\\
\#\mathscr{B}G(\mathbb{F}_{q})& =  \sum^\infty_{i_1=0} \sum^\infty_{i_2=0}\sum^\infty_{i_3=0} \sum^\infty_{i_4=0} \xi^{(i_2 + 2i_3)} q^{-(2i_1+4i_2+4i_3+6i_4+28)}\\
\zeta_{\mathscr{B}G}(t)& = \prod_{k_1=1}^\infty \prod_{k_2=1}^\infty  \prod_{k_3=1}^\infty   \prod_{k_4=1}^\infty \left(1- \xi^{(k_2+2k_3)}q^{-(2k_1+4k_2+4k_3+6k_4+13)} t \right)^{-1}
\end{align*}

\begin{remark}It is well known, that there exists another exotic infinite class of finite simple groups of Lie type, namely the Suzuki and Ree groups  constructed by Suzuki \cite{Sz} and Ree \cite{Re1}, \cite{Re2} respectively. They are described as follows:

 	\begin{enumerate}
	\item $~^2B_2(\mathbb{F}_{q^2})$ with $q^2 = 2^{2n+1}$ \cite{Sz}.
	\item $~^2G_2(\mathbb{F}_{q^2})$ with $q^2 = 3^{2n+1}$ \cite{Re1}.
	\item $~^2F_4(\mathbb{F}_{q^2})$ with $q ^2= 2^{2n+1}$ \cite{Re2}.
	\end{enumerate}

These are again twisted versions of groups of Lie type, but they cannot be interpreted directly as groups of $\F_q$-rational points of some algebraic group in the way described above. However, one could heuristically consider, for example, the group $~^2B_2$ as being an algebraic group over a ``field'' with ${2^{(n+\frac{1}{2})}}$ elements, which, of course, does not exist. On the other hand
de Medts and K. Naert \cite{MN} have recently developed a general framework in which these groups can be obtained as groups of rational points of algebraic groups over a "twisted field" by appropriately extending the category of schemes to a category of "twisted schemes".  Such a twisted field can then be interpreted like a "field with $\sqrt{\smash[b]p}$ elements". Following along this path, one could consider also a category of twisted algebraic stacks and related Lefschetz trace formula which would then allow to do similar considerations as above for classifying stacks of algebraic groups over "twisted fields". This will be part of a follow-up project.\\

{\it Acknowledgements:} The second author likes to thank the University of Michigan-SJTU Joint Institute (UM-SJTU) in Shanghai for its hospitality while parts of this article were completed. Both authors also like to thank the referee for valuable comments.

\end{remark}

\end{document}